\documentclass[11pt,reqno]{amsart}
\usepackage{amsmath,amssymb,amsthm,enumerate, hyperref}
\usepackage{bm, soul, color}
\usepackage{amsfonts, mathrsfs}
\usepackage{comment}
\usepackage{verbatim}
\usepackage[margin=1in]{geometry}

\usepackage[normalem]{ulem}

\usepackage[pdftex]{graphicx}

\newcommand{\R}{\mathbb{R}}

\newcommand{\Z}{\mathbb{Z}}
\newcommand{\N}{\mathbb{N}}

\newcommand{\mH}{\mathcal{H}}
\newcommand{\mL}{\mathcal{L}}

\newcommand{\norm}[1]{\left\lVert#1\right\rVert}
\newcommand{\wt}{\widetilde}

\newcommand{\sse}{\subseteq}
\newcommand{\inv}{^{-1}}

\newtheorem{thm}{Theorem}
\newtheorem{lem}{Lemma}
\newtheorem{prop}[thm]{Proposition}
\newtheorem{cor}[thm]{Corollary}
\theoremstyle{definition}

\theoremstyle{remark}

\author{}

\email{}

\title[]{Multiplicative diophantine approximation on planar lines with restricted denominators}

\author{Lucas Tapia}
\address{Department of Mathematics and Statistics, University of Nevada, Reno, 1664 N. Virginia St., Reno, NV 89557, USA}
\email{ltapiapitzzu@unr.edu}
\begin{document}

\begin{abstract}
    We prove a Khintchine result for convergence of a multiplicative Diophantine set with restricted denominators on an arbitrary non-degenerate line. Specifically, given sequences of real numbers $\{a_n\}_{n\in\N},\, \{b_n\}_{n\in\N},\, \{c_n\}_{n\in\N},\, \{d_n\}_{n\in\N},$ we determine convergence conditions under which the set of $x\in [0,1]$ which satisfy $\norm{a_n x +c_n} \cdot \norm{b_n x + d_n }< \psi(n) $ for infinitely many $n\in\N$ has zero Hausdorff s-measure. We also obtain an upper bound for the Hausdorff dimension in the inhomogeneous setting.
\end{abstract}

\maketitle
\section{Introduction}
Let $\N$ denote the set of positive integers, and let $\psi$ be a nonnegative real function defined on $\N$.   A central theme in metric diophantine approximation is to study, for a given subset $X$ of the $k$-dimensional real Euclidean space $\R^k$, the measure theoretic property of the set of $x\in X$  for which there exist infinitely many $ n\in\N$ such that \[\max\limits_{1\leq i\leq k} \norm{nx_i} < \psi(n),\] where $\norm{x}$ is the  distance from $x$ to the nearest integer. Here $X$ can be any interesting arithmetic/analytic/geometric object, such as algebraic varieties, manifolds, and fractal sets. 

In the simplest case when $X=[0,1]^k$, denote this limsup set in question by $S_k(\psi)$.
A classical theorem of Khintchine \cite{Kh} states that the Lebesgue measure of $S_k(\psi)$ satisfies an elegant zero-one law

 \[\lambda(S_k(\psi)) = \begin{cases}
        1,& \text{if } \displaystyle\sum_{n\in\N} \psi(n)^k = \infty \text{ and } \psi \text{ is monotonic},\\
        0,& \text{if } \displaystyle\sum_{n\in\N} \psi(n)^k < \infty.
    \end{cases}\]
    
 One may also study the problem in the multiplicative setting. Let  
    \[S_k^\times(\psi) := \{x\in [0,1]^k : \prod\limits_{1\leq i\leq k} \norm{nx_i} < \psi(n) \text{ for i.m. } n\in\N\}.\]
    Here ``i.m." stands for ``infinitely many".
   In 1962, Gallagher \cite{Gal} proved that
\begin{equation}\label{Gallagher}\lambda(S_k^\times(\psi)) = \begin{cases}
            1,& \text{ if }\displaystyle\sum_{n\in\N} \psi(n) \log(n)^{k-1} = \infty\text{ and } \psi \text{ is monotonic},\\
            0,& \text{ if }\displaystyle\sum_{n\in\N} \psi(n) \log(n)^{k-1} < \infty.
        \end{cases}\end{equation}
Furthermore, Bovey and Dodson \cite{bovey} extended this Lebesgue theoretic result to the setting of Hausdorff measures in 1978, and obtained for  $s\in (k-1, k)$ that
        \begin{equation}\label{Bovey and Dodson}\mH^s(S_k^\times(\psi)) = \begin{cases}
            \infty,& \text{ if }\displaystyle\sum_{n\in\N} n^{k-s}\psi(n)^{s-k+1} = \infty\text{ and } \psi \text{ is monotonic},\\
            0,& \text{ if }\displaystyle\sum_{n\in\N}n^{k-s}\psi(n)^{s-k+1} < \infty.
        \end{cases}\end{equation}
In spite of these measure theoretic results, a long standing conjecture of Littlewood \cite{little}, which asserts that
\[
S_2^\times\left(\frac{\varepsilon}{n}\right)=[0,1]^2 \quad\text{for every }\varepsilon>0,
\]
remains wide open. 

Many researchers have also studied the measure of $S_2^\times(\psi)\cap \mL$ for a planar line $\mL$. For instance, it is known that for almost all $(x,y)\in \mL$ we have
\begin{equation}\label{chow} \liminf\limits_{n\to\infty} n (\log n)^2 \norm{nx}\norm{ny} = 0,
 \end{equation}
thanks to Beresnevich, Haynes, and Velani \cite{BHV} in the case when $\mL$ is a coordinate line, and Chow and Yang \cite{ChowYang} in the general case. 
    Related inhomogenous results can be found in \cite{CT} and the references therein. 
  
 Other variations of the problem include  allowing for sequences more general than $\{n\}$ as well as inhomogeneous shifts. Precisely, given any four sequences $\{a_n\}$, $\{b_n\}$, $\{c_n\}$, and $\{d_n\}$,  we consider the set
 \begin{equation}\label{M2}M_2(\psi) :=\left\{(x,y)\in[0, 1]^2 : \norm{a_nx + c_n}\cdot\norm{b_n y+d_n} < \psi(n) \text{ for i.m. } n\in\N  \right\}. \end{equation}
 Unless otherwise specified, we always assume that
 $$
 a_n, b_n, c_n, d_n\in\R \quad \text{and}\quad1\le a_n\leq b_n, \quad\text{for all }n\in\N. 
 $$
   Clearly, if $a_n = b_n = n$ and $c_n = d_n = 0$ for all $ n\in\N$, the set $M_2(\psi)$ reduces back to $S_2^\times(\psi)$.
   
   In 2025, Li, Li, and Wu \cite{LLW} were able to show that if $\{a_n\}_{n\in\N}$ and $\{b_n\}_{n\in\N}$ are both sequences of distinct positive integers,  and $c_n = d_n = 0$ for all $n\in \N$, then $M_2(\psi)$ has Hausdorff dimension $\min\{ 2, 1+\tau\} $, where 
 \begin{equation}\label{tau}\tau = \inf  \left\{ s> 0 : \sum_{n\in\N}b_n \left(\frac{\psi(n)}{b_n}\right)^s <\infty \right\}.  \end{equation}
 The methods in \cite{LLW} do not allow for a version of the lower bound of $\dim_\mathrm{H} M_2(\psi)$ in the case that $a_n$ or $b_n$ are not integers, but we present an identical upper bound in the inhomogeneous real case in Proposition \ref{dimension R^2}. 
Essentially, the upper bound for the Hausdorff dimension of $M_2(\psi)$ is blind to any inhomogeneous shifts, and although there is no reason to suspect the lower bound is any different, this remains an open problem. Furthermore, letting $\{a_n\}$ and $\{b_n\}$ attain real values is a superficial allowance in the case of $M_2(\psi)$, as the counting results are identical for integers.

It is not hard to obtain a Khintchine type statement for the set $M_2(\psi)$, as we will demonstrate in Theorem \ref{dimension R^2} and the comment therein. It is, however, of great interest to investigate the intersection of  $M_2(\psi)$ with the line $y=x$. Let
\begin{equation}\label{prototype}
    M(\psi) :=\{x\in[0, 1] : \norm{a_nx + c_n}\cdot\norm{b_n x + d_n} < \psi(n) \text{ for i.m. } n\in\N  \}.
\end{equation}

\begin{thm} \label{integer_shift}
    Suppose that $\{a_n\}$ and $\{b_n\}$ are restricted to be sequences of positive integers. 
    \begin{enumerate}[(a)]
    \item
    For $s\in(0,1)$, we have
    $$
\mathcal{H}^s(M(\psi))=0
$$
under the condition that
    \begin{equation}\label{integerconv}
\sum_{n\in\N}b_n \left(\frac{\psi(n)}{b_n}\right)^s +\sum_{n\in\N} \gcd (a_n, b_n) \left(\frac{\psi(n)}{a_nb_n}\right)^{s/2}<\infty.
    \end{equation}
    \item 
 In the case of the Lebesgue measure, we have
    $$
\lambda(M(\psi))=0
$$
under the condition that
    \begin{equation*}
\sum_{n\in\N}\psi(n)\log\frac{1}{\psi(n)} +\sum_{n\in\N} \gcd (a_n, b_n) \left(\frac{\psi(n)}{a_nb_n}\right)^{1/2}<\infty.
    \end{equation*}
    \end{enumerate}
In particular, it follows that
 \[
    \dim_\mathrm{H}M(\psi) \leq \min\{1, \tau\},
    \] 
    where $\tau$ is the infimum of those $s$ for which \eqref{integerconv} holds.
\end{thm}

In 2021, Zhang and Lü \cite{ZL} proved for $a_n = 2^n$, $b_n = 3^n$, $c_n=0$, $d_n = 0$, and $\psi$ monotonic, that
\begin{equation}\label{LLW dimension}
\dim_\mathrm{H} M(\psi)=\min \{1, \tau\}
\end{equation}
with the same $\tau$ as given in Theorem \ref{integer_shift}. This result has subsequently been extended by Li, Li, and Wu \cite{LLW} to allow both $\{a_n\}$ and $\{b_n\}$ to be sequences of distinct positive integers.

In our next main result, we manage to generalize Theorem \ref{integer_shift} to real sequences $\{a_n\}$ and $\{b_n\}$ at the expense of assuming a more restrictive convergence condition on $\psi$.

\begin{thm}\label{dimension}
   We have
    \[\dim_\mathrm{H} M(\psi) \leq \min\{1, \tau\},\]
    where \begin{equation*}\label{rho}\tau = \inf\left\{s
    > 0: \sum_{n\in\N} \left[  b_n\left(\frac{\psi(n)}{b_n}\right)^s + a_n\left(\frac{\psi(n)}{a_nb_n}\right)^{s/2}   \right]<\infty\right\}. \end{equation*}
\end{thm}
We obtain this result as a consequence of a convergence implication in the spirit of \eqref{Bovey and Dodson}; see Proposition \ref{conv case} below. Incidentally, we also obtain a similar Lebesgue convergence result, as in \eqref{Gallagher}; see Proposition \ref{Lebesgue}.
It is unknown if $\dim_\mathrm{H}(M(\psi)) \geq \min\{1, \tau\}$ in general.

As a corollary, we restrict our attention to specific exponential values of $a_n$ and $b_n$. 
Here and throughout the paper, the natural logarithmic function $\log x$ is refined to be 1 for $x\le e$.

   \begin{cor}\label{exponent}
   Suppose that $\{a_n\}=\{a^n\}$ and $\{b_n\}=\{b^n\}$ for some $a,b\in \R$ with $1<a<b$. 
        For $  \max\left\{2-\frac{\log b}{\log a},0\right\}<s<1$, we have 
		\[\sum_{n\in\N} b^n \left( \frac{\psi(n)}{b^n}\right)^s < \infty  \implies \mH^s(M(\psi)) = 0.\]
        In particular, if $a^2 \leq b$, the conclusion holds for all $s\in (0, 1)$.
    \end{cor}

The organization of the paper is as follows. As is custom, we recall the definition of the Hausdorff measure in section 2. In section 3, we prove some counting results (Lemma \ref{key}) that transfer to a nontrivial covering (Lemma \ref{cover}), then estimate the measure of the cover (Lemmas \ref{Edelta measure} , \ref{edelta measure lebesgue}).  In section 4, we use the counting to prove convergence results used in the proof of Theorem \ref{dimension}, and then prove the rest of the Theorems. As we are not concerned with the complementary divergence statements 
and our counting result holds for fixed $n$,  
we do not need $\psi$ to be monotonic.

Throughout, we use the notation $e(x)$ to mean $e^{2\pi i x},$ and we interchangeably utilize the Vinogradov and Bachmann-Landau notations $f\ll g$ or $f=O(g)$ if there is a constant $C>0$ such that $|f| \leq Cg.$ The notation $f\asymp g$ is used when $f\ll g$ and $g\ll f.$ Unless otherwise noted, all implied constants are absolute, and in particular do not depend on $n$. We also use the floor function $\lfloor x\rfloor = \max\{n \in \N : n\leq x\}$ and the ceiling function $\lceil x \rceil = \min \{n\in \N : x\leq n\}$.

\section{Definition of the Hausdorff measure}

For completeness, we recall the definition of the Hausdorff measure and dimension. Let $ F\sse \R^k $. For $ \rho>0 $, a $ \rho $-cover of $ F $ is a countable collection of balls $ \{B_i\} $, where $  B_i $ is a ball with radius $ r_i<\rho$, such that $ F\sse \bigcup_i B_i $.
		For $ s\ge0 $, define \[\mH^s_\rho(F) = \inf \sum_ir_i^s,\] where the infimum is taken over all $ \rho $-covers of $ F $.
		Then the Hausdorff $s$-measure of $F$ is defined as  \[\mH^s(F) = \sup\limits_{\rho>0}\mH^s_\rho(F) = \lim\limits_{\rho\to 0}\mH^s_\rho(F),\]
which exists as a finite number or approaches $\infty$. Moreover, the Hausdorff dimension of $F$ is therefore defined by
$$
\dim_\mathrm{H}(F)=\inf\{s:\mathcal{H}^s(F)=0\}=\sup\{s:\mathcal{H}^s(F)=\infty\}.
$$

We end this section by stating the Borel-Cantelli lemma for Hausdorff measures. 
\begin{lem}[\cite{BV} Lemma 3.10]\label{cantelli}
Let $\{H_j\}$ be a countable collection of hypercubes in $\R^k$ and suppose
that for some $s>0$, 
\[\sum_{j} \text{diam}(H_j)^s < \infty.\]
Then $ \mathcal{H}^s(\limsup H_j) = 0 $.
	\end{lem}

\section{Preliminary Estimates}
\subsection{The real case}
Throughout this section, we assume that $a,b,c,d\in\R$ and $1\le a\le b$.

\subsubsection{A covering result and lattice counting}
       
For $0<\eta, \xi < 1$, consider the set 
	\begin{align*}
		F(\eta, \xi ) &= \{x\in \left[ 0,1 \right] : \norm{ax+c} < \eta,\, \norm{bx+d} < \xi\}.
	\end{align*} 
Then \[F(\eta, \xi) \subseteq \bigcup_{\substack{(p, q)\in \Z^2\\ \lfloor c\rfloor\leq p\leq \lceil a +c \rceil\\ \lfloor d \rfloor\leq q \leq \lceil b+d \rceil }} B\left(\frac{p-c}{a}, \frac{\eta}{a}\right)\cap B\left(\frac{q-d}{b}, \frac{\xi}{b}	 \right)\,,\]
which can be covered by  
    \begin{equation*}\label{N(eta, delta)}N(\eta, \xi):=\#\left\{(p, q)\in \Z^2 : \lfloor c\rfloor\leq p\leq \lceil a +c \rceil,\, \lfloor d \rfloor\leq q \leq \lceil b+d \rceil  , \left\lvert \frac{p-c}{a} - \frac{q-d}{b} \right\rvert < \frac{\eta}{a}+ \frac{\xi}{b}\right\}\end{equation*}
    many intervals of length $2\min\left\{\frac\eta{a},\frac{\xi}{b}\right\}$.

Write
   \begin{equation}\label{L}
   L= \max\left\{1, \frac{\log{b}}{a}\right\}.
   \end{equation}

\begin{lem}\label{cover}
    The set $ F(\eta, \xi) $ can be covered by at most $ O(b\eta L+aL)$ many intervals of length $ \min\left\{\frac{\eta}{a}, \frac{\xi}{b}\right\} .$
\end{lem}

In view of the above discussion, to prove Lemma \ref{cover}, it suffices to show the next lattice counting estimate.

\begin{lem}\label{key}
		We have \begin{equation*}\label{counting}N(\eta, \xi) \ll ( b\eta  +  a)L,\end{equation*}
         where the implicit constant is absolute.
	\end{lem}

    Before we embark on the proof of Lemma \ref{key}, we need to state one more lemma.

    \begin{lem}[Erd\H os-Tur\'an Inequality (\cite{mo}, Theorem 1.1)]\label{erdos}
Let $\mathcal{U}$ be a sequence of $Q$ elements in $\mathbb{T}:=\mathbb{R}/\mathbb{Z}$. Then for any closed interval $I\subseteq\mathbb{T}$ and positive integer $K$, the  discrepancy  function
\[ D(\mathcal{U},I) := \# (\mathcal{U}\cap I) -|I| Q\]
satisfies the estimate
\[
 |D(\mathcal{U},I)| \leq \frac{Q}{K+1} +   2\sum_{k=1}^K \left(\frac{1}{K} + \min\left(|I|, \frac{1}{\pi k}\right)\right) \left| \sum_{u\in\mathcal{U}} e(ku)\right|.
 \]
\end{lem}
\begin{proof}[Proof of Lemma \ref{key}]
    
    We will consider the case where $\eta + \frac{a}{b}\xi> \frac{1}{2}$ first. 
    For every value of $q$, there are at most four values of $p$ that are counted in $N(\eta, \xi)$, since
    \begin{equation*}
        \left| p - c - \frac{aq-ad}{b}\right| < \eta + \frac{a}{b}\xi < 2.
    \end{equation*}
    Therefore \begin{equation*}
        N(\eta, \xi) \leq 4(b+2)  \ll b\eta + a.
    \end{equation*}

    It remains to deal with the case $\eta + \frac{a}{b}\xi\le \frac{1}{2}$. To that end, denote
$$
\delta = \eta + \frac{a}{b}\xi.
$$
    Then $N(\eta,  \xi)$ is bounded above by 
    \begin{align*}\#\left\{ q\in \Z : \lfloor d \rfloor\leq q \leq \lceil b+d \rceil ,  \norm{\frac{aq}{b}-\frac{ad}{b} + c}< \delta\right\}\,.
	\end{align*}
    Now consider the sequence 
    \[\mathcal{U}:=\{u_q\}_{q=1}^Q\]
where
$$u_q := \frac{a}{b} (q + \lfloor d\rfloor - 1) - \frac{ad}{b} + c,$$
and $$Q= \lceil b + d\rceil -\lfloor d\rfloor + 1.$$
It is clear that $$Q\asymp b.$$
Next  we apply the Erd\H os-Tur\'an inequality (Lemma \ref{erdos}) to $\mathcal{U}$ with 
$$I=[-\delta,\delta]$$
and
$$
K=\left\lfloor\frac{b}{a}\right\rfloor.
$$
On observing that	
\begin{align*}
	\sum_{k=1}^K\left|\sum_{q=1}^Qe(ku_q)\right|&= \sum_{k=1}^K  \left| \sum_{q=1}^{Q}e\left(k\left(\frac{a}{b} (q + \lfloor d\rfloor - 1) - \frac{ad}{b} + c\right)\right)\right| \\
   &=\sum_{k=1}^K\left|\sum_{q=1}^{Q}e\left(\frac{kaq}{b } \right)\right|\\
   &\ll \sum_{k=1}^K \min\left\{Q, \norm{\frac{ka}{ b}}\inv\right\}\\
   &\ll Q+\sum_{k=1}^K\frac{b}{ak}\\
   &\ll Q+K\log K,
\end{align*}
we obtain
\begin{align*}
			|D(\mathcal{U},I)|
			&\ll \frac{Q}K+\left(\frac1K+\delta\right)\Big(Q+K\log K \Big) \\
            &\ll  \left(\frac{b}K+\delta b\right)\left(\frac{Q}b+\frac{K}{b}\log K \right)\\
            &\ll (a+\delta b)\left(1+\frac1a\log \frac{b}{a}\right)\\
            &\ll (a+\delta b)L,
		\end{align*}
where in the last line there is no loss in replacing $\log \frac{b}{a}$ by $\log b$ since
$$
0\le \frac1a\left(\log b-\log \frac{b}{a}\right)\le\frac{\log a}a\le1.
$$
So 
$$
N(\eta,\xi)\ll \delta Q+|D(\mathcal{U},I)|\ll (b\eta +a)L.
$$
\end{proof}

    \subsubsection{Hausdorff Measure Estimates}
We are now poised to estimate the Hausdorff measure of the set
          \begin{equation*}\label{Edelta}	 E(\delta) =\{	x\in [0,1] : \norm{a x+c} \cdot \norm{bx+d} < \delta^2\}.\end{equation*}
        
    \begin{lem}\label{Edelta measure}
       For $s\in (0, 1)$, $\delta\in(0,1/2]$, and $\rho>1$, we have
        \begin{equation*}\label{Edelta measure equation}\mH_\rho^s(E(\delta)) \ll b\left(\frac{\delta^2}{b}\right)^sL +  a\left(\frac{\delta^2}{ab}\right)^{s/2}L,\end{equation*}
        where the implicit constant only depends on $s$, and $L$ is given in \eqref{L}.
    \end{lem}
    \begin{proof}
        Decompose $ E(\delta) = A(\delta) \cup B(\delta)\cup C(\delta) $, where 
		\begin{align*}
			A(\delta) &= \{	x\in [0,1] : \norm{ax+c} <\delta,  \norm{bx+d} < \delta\},\\
			B(\delta) &= \{	x\in [0,1] : \norm{ax+c} \ge \delta, \norm{ax+c} \cdot \norm{bx+d} < \delta^2\},\\	
			C(\delta) &= \{	x\in [0,1] : \norm{bx+d} \ge \delta, \norm{ax+c} \cdot \norm{bx+d} < \delta^2\}.
		\end{align*}
		In view of $ A(\delta) = F(\delta, \delta) $,  it follows from Lemma $ \ref{cover} $ that
        \begin{align}
            \mH_\rho^s(A(\delta)) &\ll (b\delta L+ aL)\min\left\{\frac{\delta}{a}, \frac{\delta}{b}\right\}^s\nonumber\\ 
            &\ll (b\delta L + aL)\left(\frac{\delta}{b}\right)^s\nonumber\\
            &= b^{1-s} \delta^{1+s}L + ab^{-s}\delta^sL.\label{A}
        \end{align}
        Let $J= \{j \geq 0 : 2^{j+1}\delta <1\} $, $J_1 = \{ j\in J : 2^{2j}\leq \frac{b}{a}\}$, and $J_2 = \{ j\in J : 2^{2j}\geq \frac{b}{a}\}$.
		Next, note that 
		\begin{align*}
			B(\delta) &\sse \bigcup_{j\in J} \left\{	x: 2^j\delta \leq \norm{ax+c} < 2^{j+1}\delta, \norm{ax+c}\cdot \norm{bx+d}< \delta^2	\right\}\\
			&\sse \bigcup_{j\in J} \left\{	x:  \norm{ax+c} < 2^{j+1}\delta,  \norm{bx+d}< 2^{-j}\delta	\right\}\\
			&=  \bigcup_{j\in J} F(2^{j+1}\delta,2^{-j}\delta ),
		\end{align*} 
        and similarly, 
        \begin{align*}
			C(\delta) &\subseteq 
			 \bigcup_{j\in J} F(2^{-j}\delta, 2^{j+1}\delta ).
		\end{align*}	
        Therefore another application of Lemma \ref{cover} yields that
        \begin{align}
            \mH_\rho^s(B(\delta))  \ll& \sum_{j\in J} \left(b2^{j+1}\delta L+aL\right)\left(\frac{2^{-j}\delta}{b}\right)^s\nonumber \\ 
            \ll& b\delta\left(\frac{\delta}{b}\right)^sL\sum_{j\in J } (2^{1-s})^j + a\left(\frac{\delta}{b}\right)^s L\sum_{j\in J}(2^{-s})^j\nonumber\\ 
            \ll& b^{1-s}\delta^{2s}L+ ab^{-s}\delta^sL, \label{B}
        \end{align}
        and
        \begin{align}
            \mH_\rho^s(C(\delta)) \ll& \sum_{j\in J_1}\left(b2^{-j}\delta L+aL\right)\left(\frac{2^{j+1}\delta}{b}\right)^s   
            + \sum_{j\in J_2} aL\left(\frac{2^{-j}\delta}{a}\right)^s\nonumber\\
            \ll & b\delta \left(\frac{\delta}{b}\right)^s L\sum_{j\in J_1} (2^{-1+s})^j + a\left(\frac{\delta}{b}\right)^s L\sum_{j\in J_1}(2^s)^j  \nonumber\\
            &+ a\left(\frac{\delta}{a}\right)^sL\sum_{j\in J_2} (2^{-s})^j\nonumber\\
            \ll & b^{1-s}\delta^{1+s}L  + a^{1-\frac{s}{2}}b^{-\frac{s}{2}}\delta^s L. \label{C}
        \end{align}

        Combining \eqref{A}, \eqref{B}, and \eqref{C}, we have 
        \begin{align*}
            \mH_\rho^s(E(\delta)) \ll& \mH_\rho^s(A(\delta)) + \mH_\rho^s(B(\delta)) + \mH_\rho^s(C(\delta))\\
            \ll& (b^{1-s} \delta^{1+s}L + ab^{-s}\delta^sL) + (b^{1-s}\delta^{2s}L+ ab^{-s}\delta^sL) + (b^{1-s}\delta^{1+s}L + a^{1-\frac{s}{2}}b^{-\frac{s}{2}}\delta^s L)\\
            \ll& b\left(\frac{\delta^2}{b}\right)^sL + a\left(\frac{\delta^2}{ab}\right)^{s/2}L.
        \end{align*}
    \end{proof}

    \begin{lem}\label{edelta measure lebesgue} For $0<\delta\le1/2$, we have
        \begin{equation*}\label{Edelta measure lebesgue eq}
        \lambda(E(\delta)) \ll   \delta^2 L\log\frac{1}{\delta} + \left(\frac{a}{b}\right)^{1/2}\delta L  ,
    \end{equation*}
    where the implicit constant is absolute, and $L$ is given in \eqref{L}.
    \end{lem}
    \begin{proof}
        We decompose $E(\delta)$ as in Lemma \ref{Edelta measure}, and proceed with $s=1$, with the modifications 
        \begin{equation}\label{B lebesgue estimate}
            \mH_\rho^1(B(\delta)) \ll \delta^{2}L |J| + ab^{-1} \delta L,
        \end{equation}
        and 
        \begin{equation}
            \mH_\rho^1(C(\delta)) \ll \delta^{2}L |J_1| + a^{\frac{1}{2}}b^{-\frac{1}{2}}\delta L\label{C lebesgue estimate} .
        \end{equation}
    The result follows from \eqref{A}, \eqref{B lebesgue estimate}, \eqref{C lebesgue estimate}, and the fact that $\mathcal{H}^1$ and $\lambda$ are comparable measures. 
    \end{proof}

\subsection{The integer case}

In this section, we obtain a significant improvement of Lemma \ref{key} by further assuming that $a$ and $b$ are integers. This leads to superior Hausdorff and Lebesgue measure estimates of $E(\delta)$. 
    \begin{lem}\label{integer key}
      If  $a, b\in \N$, then
        \[N(\eta, \xi) \ll   b\eta + (a, b).\]
        where the implicit constant is absolute.
    Furthermore, for any $\rho>1$, \begin{equation*}
            \mH_\rho^s(E(\delta)) \ll b\left(\frac{\delta^2}{b}\right)^s + (a, b)\left(\frac{\delta^2}{ab}\right)^{s/2},
        \end{equation*}
        where the implicit constant depends only on $s$.
    For completion, \begin{equation*}
            \lambda(E(\delta)) \ll \delta^2\log \frac{1}{\delta} + (a, b)\left(\frac{\delta^2}{ab}\right)^{1/2},
        \end{equation*}
        where the implicit constant is absolute.
    \end{lem}
    \begin{proof}
        We follow a similar strategy as in the proof of Lemma \ref{key}, up to the step where we consider the exponential sum present in the Erd\H os-Tur\'an inequality (Lemma \ref{erdos}).
        We first consider the case where $\eta + \frac{(a, b)}{b}\xi> \frac{1}{2}$ to obtain 
     \begin{equation*}
        N(\eta, \xi) \leq 4(b+2)  \ll  b\eta + (a, b).
    \end{equation*}

    It remains to deal with the case where $\eta + \frac{(a, b)}{b}\xi\le \frac{1}{2}$. Denote \[\delta = \eta + \frac{(a, b)}{b}\xi.\] By choosing $K = \frac{b}{(a, b)}$, observe that 
        \begin{align*}
            \sum_{k=1}^K\left|\sum_{q=1}^{b+\lceil d\rceil-\lfloor d\rfloor + 1}e\left(ku_q\right)\right| &\ll  \sum_{k=1}^K \left(2 + \left|\sum_{q=1}^{b}e\left(k\left(\frac{a}{b}\left(q + \lfloor d\rfloor - 1\right) - \frac{ad}{b} + c\right)\right)\right|\right)\\
            &\ll K + \sum_{k=1}^K \left|(a, b)\sum_{q=1}^{\frac{b}{(a, b)}}e\left(\frac{kq\frac{a}{(a, b)}}{\frac{b}{(a, b)}}\right)\right|\\
            &\ll K + \sum_{k=1}^K b\chi_{\frac{b}{(a, b)}|k}(k)\\
            &\ll b.  
        \end{align*}
        Then 
        \begin{equation*}
            |D(\mathcal{U},I)|	\ll (a, b) + b\delta.
        \end{equation*}
        
        Inserting the above bound in lieu of Lemma \ref{key} into the proofs of Lemmas \ref{Edelta measure} and \ref{edelta measure lebesgue} yields the remaining estimates. 
       
    \end{proof}

\section{Convergence Results}
    For each $ n\in \N $, denote the set 
	\[	 E_n(\delta) =\{	x\in [0,1] : \norm{a_n x+c_n} \cdot \norm{b_nx+d_n} < \delta^2\},\]
    so that \[M(\psi) = \limsup\limits_{n\to\infty} E_n(\psi(n)^{1/2})\,.\]
    Also, denote \begin{equation*}\label{Ln}
   L_n= \max\left\{1, \frac{\log{b_n}}{a_n}\right\}.
   \end{equation*}
    \begin{prop}\label{conv case}
    Let $\{a_n\}_{n\in\N},\, \{b_n\}_{n\in\N},\,\{c_n\}_{n\in\N},\, \{d_n\}_{n\in\N}$  be real sequences such that $1\le a_n\le b_n$ for all $n\ge1$, and let $s\in (0, 1).$
    If \begin{equation}\label{s-measure convergence}\sum_{n\in\N}  \left[ b_n \left(\frac{\psi(n)}{b_n}\right)^s +  b_n \left(\frac{\psi(n)}{b_n}\right)^s\frac{\log{b_n}}{a_n} + a_n\left(\frac{\psi(n)}{a_nb_n}\right)^{s/2} +\left(\frac{\psi(n)}{a_nb_n}\right)^{s/2}\log{b_n}\right]< \infty ,\end{equation}
    then \[ \mH^s(M(\psi)) = 0\,.\]
\end{prop} 
    \begin{proof}
        In view of the convergence of $\sum \psi(n)^s b_n^{1-s}$, we must have $\psi(n)\to 0$. Hence, it can be assumed that $\psi(n)\le \frac{1}{4}$ for all but finitely many $n$.
        By the Borel-Cantelli Lemma, to find $ \mH^s(M(\psi)) $, it suffices to estimate $ \mH^s(E_n(\delta)) $ for $ \delta^2 = \psi(n) $.\\
        By Lemma \ref{Edelta measure}, for $\rho>1$, we have
        \begin{equation*}
			\sum_{n\in\N} \mH_\rho^s(E_n(\psi(n)^{1/2})) 
			\ll   \sum_{n\in\N}\left[ b_n \left( \frac{\psi(n)}{b_n}\right)^sL_n + a_nL_n\left(\frac{\psi(n)}{a_nb_n}\right)^{s/2}\right]
                <  \infty,
		\end{equation*}
        and so $\mH_\rho^s(M(\psi)) = 0$. It follows that 
  \[\mH^s(M(\psi)) = 0.\] 
    \end{proof}

\begin{prop}\label{Lebesgue}
    Let $\{a_n\},\, \{b_n\},\,\{c_n\},\,\{d_n\}$ be as in Proposition \ref{conv case} .
    If 
    \begin{align}
        \sum_{\substack{n\in\N\\ \psi(n) > 0}}&\psi(n)\log\frac1{\psi(n)} +\frac{\psi(n)}{a_n}\log{b_n}\log\frac1{\psi(n)}+\left(\frac{\psi(n)a_n}{b_n}\right)^{1/2}+
        \left(\frac{\psi(n)}{a_nb_n}\right)^{1/2} \log{b_n}
        < \infty,\label{Lebesgue measure convergence}
    \end{align}
    then \[\lambda(M(\psi)) = 0.\]    
\end{prop}
\begin{proof}
    First, observe that by \eqref{Lebesgue measure convergence}, $\psi(n) \geq \frac{1}{4}$ for only finitely many $n$, and so we may take $n$ large enough to assume \[\log(\psi(n)\inv) \geq 1.\]
    Taking $\delta^2 = \psi(n)$, by Lemma \ref{edelta measure lebesgue}, 
    \[\lambda(E_n(\psi(n)^{1/2}) \ll \psi(n)\log\frac1{\psi(n)} +\frac{\psi(n)}{a_n}\log{b_n}\log\frac1{\psi(n)}+\left(\frac{\psi(n)a_n}{b_n}\right)^{1/2}+
        \left(\frac{\psi(n)}{a_nb_n}\right)^{1/2} \log{b_n}. \]
    It follows by the Borel-Cantelli lemma that
    \[\lambda(M(\psi)) = 0.\]
        
\end{proof}

Theorem \ref{integer_shift} follows in a similar manner from Lemma \ref{integer key}.

 \begin{prop}\label{dimension R^2}
    Let $\{a_n\}_{n\in\N},\, \{b_n\}_{n\in\N},\,\{c_n\}_{n\in\N},\,\{d_n\}_{n\in\N}$ be sequences in $\R$, with $1\le a_n \leq b_n$ for all $n\geq 1$. For $s\in (0, 1)$, if \begin{equation}\label{R2measure}\sum_{n\in\N} b_n \left(\frac{\psi(n)}{b_n}\right)^s <\infty ,\end{equation} then 
    \[\mH^{1+s}(M_2(\psi)) = 0.\]
    In the Lebesgue setting, if \begin{equation*}\sum_{n\in\N}  \psi(n)\log \frac{1}{\psi(n)} <\infty ,\end{equation*} then 
    \[\mH^2(M_2(\psi)) = 0.\]
    In particular,
    \[\dim_\mathrm{H}M_2(\psi) \leq \min\{2, 1+\tau\},\]
    where $\tau$ is as in \eqref{tau}.
\end{prop}
\begin{proof}
    Denote \[E_n'(\delta) =\{	(x, y)\in [0,1]^2 : \norm{a_n x+c_n} \cdot \norm{b_ny+d_n} < \delta^2\},\]
    so that \[M_2(\psi) = \limsup\limits_{n\to\infty} E_n'(\psi(n)^{1/2}).\]
    For $\eta, \xi \in \left(0, 1\right)$, denote 
    \begin{align*}
		F_n'(\eta, \xi ) &= \{(x, y)\in \left[ 0,1 \right]^2 : \norm{a_nx+c_n} < \eta,\, \norm{b_ny+d_n} < \xi\}.\\ 
	\end{align*} 
    Then \[F_n'(\eta, \xi) \subseteq \bigcup_{\substack{(p, q)\in \Z^2\\ \lfloor c_n\rfloor\leq p\leq \lceil a_n +c_n \rceil\\ \lfloor d_n \rfloor\leq q \leq \lceil b_n+d_n \rceil }} B\left(\frac{p-c_n}{a_n}, \frac{\eta}{a_n}\right)\times B\left(\frac{q-d_n}{b_n}, \frac{\xi}{b_n}	 \right)\,,\]
which can be covered by \[O\left(a_nb_n\frac{\max\left\{\frac{\eta}{a_n}, \frac{\xi}{b_n}\right\}}{\min\left\{\frac{\eta}{a_n}, \frac{\xi}{b_n}\right\}}\right)\] many squares of side length \[\min\left\{\frac{\eta}{a_n}, \frac{\xi}{b_n}\right\}.\]
Then for any $\rho>2$,
\begin{align}
    \mH_\rho^{1 +s}(F_n'(\eta, \xi)) &\ll a_nb_n\max\left\{\frac{\eta}{a_n}, \frac{\xi}{b_n}\right\}\min\left\{\frac{\eta}{a_n}, \frac{\xi}{b_n}\right\}^{s}\nonumber\\
    &\ll \eta\xi \min\left\{\frac{\eta}{a_n}, \frac{\xi}{b_n}\right\}^{s-1}.\label{R2 counting}
\end{align}
    For the following calculations, let $\delta \in (0, 1/2]$.
    Denote
    \begin{align*}
			A_n'(\delta) &= \{	(x, y)\in [0,1]^2 : \norm{a_nx+c_n} <\delta,  \norm{b_ny+d_n} < \delta\},\\
			B_n'(\delta) &= \{	(x, y)\in [0,1]^2 : \norm{a_nx+c_n} \geq \delta, \norm{a_nx+c_n} \cdot \norm{b_ny+d_n} < \delta^2\}\text{, and}\\	
			C_n'(\delta) &= \{	(x, y)\in [0,1]^2 : \norm{b_ny+d_n} \geq \delta, \norm{a_nx+c_n} \cdot \norm{b_ny+d_n} < \delta^2\}.
	\end{align*}
    We first consider the case when $s\in (0, 1)$. By \eqref{R2 counting}, with $J$, $J_1$, and $J_2$ as in the proof of Lemma \ref{Edelta measure}, and for any $\rho>2$,
    \begin{align*}
        \mH_\rho^{1+s}(A_n'(\delta)) &= \mH_\rho^{1+s}(F_n'(\delta, \delta))\\
        &\ll \delta^2\min\left\{\frac{\delta}{a_n}, \frac{\delta}{b_n}\right\}^{s-1}\\
        &= b_n^{1-s}\delta^{1+s},
    \end{align*}
    \begin{align*}
        \mH_\rho^{1+s}(B_n'(\delta)) &\leq \sum_{j\in J} \mH_\rho^{1+s}(F_n'(2^{j+1}\delta, 2^{-j}\delta))\\
        &\ll b_n^{1-s}\delta^{1+s} \sum_{j\in J}\left(2^{1-s}\right)^j\\
        &\ll b_n^{1-s}\delta^{2s},
    \end{align*}
    and 
    \begin{align*}
        \mH_\rho^{1+s}(C_n'(\delta)) &\leq \sum_{j\in J}\mH_\rho^{1+s}(F_n'(2^{-j}\delta, 2^{j+1}\delta))\\
        &\ll \delta^2\left[\sum_{j\in J_1}\left(\frac{2^j \delta}{b_n}\right)^{s-1}+\sum_{j\in J_2}\left(\frac{2^{-j}\delta}{a_n}\right)^{s-1}\right]\\
        &\ll b_n^{1-s}\delta^{1+s} + a_n^{1-s}\delta^{2s}.
    \end{align*}
    Then \[\mH_\rho^{1+s}(E_n'(\delta)) \ll b_n^{1-s}\delta^{2s},\] from which \eqref{R2measure}, and therefore the first result, follows.
    For the second result, proceed as above with $s=1$, and the modification 
    \[\mH_\rho^{2}(B_n'(\delta))+\mH_\rho^{2}(C_n'(\delta))\ll \delta^2 |J|.\]
\end{proof}

    \begin{proof}[Proof of Corollary \ref{exponent}]
        Clearly, $1\le a^n\le b^n$, and for $n$ large enough,
        \[a_n = a^n > n \gg  n\log b = \log b_n, \] 
        so $L_n = 1$. Consider the function \[\wt\psi(n) := \max\left\{\psi(n), \left(\frac{a_n}{b_n}\right)^{\frac{2-s}{s}}\right\},\]
        and note that \[ a_n\left(\frac{\wt\psi(n)}{a_nb_n}\right)^{s/2} \leq b_n\left(\frac{\wt\psi(n)}{b_n}\right)^s.\]
        Then \[\limsup\limits_{n\to\infty}E_n(\psi(n)^{1/2})\sse \limsup\limits_{n\to\infty}E_n(\wt\psi(n)^{1/2}),\]
        and for any $\rho>1$,
        \begin{align*}
            \sum_{n\in\N}\mH_\rho^s(E_n(\wt\psi(n)^{1/2})) \ll& \sum_{n\in\N} b_n \left(\frac{\wt\psi(n)}{b_n}\right)^s\\
            \ll & \sum_{n\in\N} b_n \left(\frac{\psi(n)}{b_n}\right)^s + \sum_{n\in\N} b_n \left(\frac{\left(\frac{a_n}{b_n}\right)^{\frac{2-s}{s}}}{b_n}\right)^s\\
            = & \sum_{n\in\N} b_n \left(\frac{\psi(n)}{b_n}\right)^s + \sum_{n\in\N} a_n^{2-s}b_n^{-s}.
        \end{align*}
        Lastly,
            \[\sum_{n\in\N} a_n^{2-s}b_n^{-1}  = \sum_{n\in\N} \left(\frac{a^{2-s}}{b}\right)^n,\]
            which converges when $\frac{a^{2-s}}{b}<1$, or equivalently when $s>2-\frac{\log b}{\log a}.$
    \end{proof}

\end{document}